\documentclass[twoside]{amsart}
\usepackage{amsmath,amsthm,amssymb,amscd,bm}
\usepackage{enumerate}

\numberwithin{equation}{section}

\newtheorem{thm}{Theorem}[section]
\newtheorem{theorem}[thm]{Theorem}
\newtheorem{proposition}[thm]{Proposition}
\newtheorem{lemma}[thm]{Lemma}
\newtheorem{corollary}[thm]{Corollary}

\theoremstyle{definition}

\theoremstyle{remark}
\newtheorem{remark}{Remark}[section]

\DeclareMathOperator{\grad}{grad}
\DeclareMathOperator{\Vol}{Vol}
\DeclareMathOperator{\dVol}{dVol}
\DeclareMathOperator{\re}{Re}
\title{A Uniqueness Theorem for Gluing Special Lagrangian Submanifolds}
\author{Yohsuke Imagi}
\date{}
\address{
Department of Mathematics,
Faculty of Science,
Kyoto University,
Kyoto,
Japan}
\email{imagi@math.kyoto-u.ac.jp}
\thanks{Supported by Grant-in-Aid for JSPS fellows(22-699)}

\begin{document}
\maketitle
\section{Introduction}
\label{introduction}
Let $M_1$, $M_2$ be special Lagrangian submanifolds of a Calabi--Yau manifold $N$, and suppose $M_1$
intersects $M_2$ transversally at a point $P$.
One can construct another special Lagrangian submanifold $M$
by gluing a Lawlor neck~\cite{Lawlor} into $M_1\cup M_2$ at $P$;
see Butscher~\cite{Butscher}, D. Lee~\cite{D. Lee}, Y. Lee~\cite{Y. Lee}, and Joyce~\cite{Joyce}.
By construction, $M$ is close to the Lawlor neck near $P$, and to $M_1\cup M_2$ away from $P$. 
Here is a problem: uniqueness
of a special Lagrangian submanifold
which is close to the Lawlor neck near $P$ and to $M_1\cup M_2$ away from $P$. 
The main result of this paper is a uniqueness theorem in the case
where $M_1$, $M_2$ are flat special Lagrangian tori of real dimension $3$, and $N$ is
a flat complex torus of complex dimension $3$; see Theorem~\ref{uniqueness} for the precise statement.
The author plans to prove a uniqueness theorem for more general $M_1$, $M_2$ in the sequel of this paper.

In section~\ref{statement} we make statement of the key step to the proof of the main result.
Section~\ref{lemma section 1} and Section~\ref{lemma section 2}
provide what we shall need in the proof of Theorem~\ref{theorem}.
In Section~\ref{proof section} we prove Theorem~\ref{theorem} with the
help of results in Section~\ref{lemma section 1} and Section~\ref{lemma section 2}.
In Section~\ref{main result} we state the main result, and prove it
by making a direct use of Theorem~\ref{theorem}.

Theorem~\ref{theorem} is similar to Simon's theorem~\cite[Theorem~5, p563]{Simon},
which was originally applied to the unique tangent cone problem for minimal submanifolds.
The proof of Theorem~\ref{theorem} is almost similar to Simon's.
There is however a significant difference between
Lemma~\ref{extension lemma} and Simon's lemma~\cite[Lemma~3, p561]{Simon}.
The condition~\eqref{energy condition} in Lemma~\ref{extension lemma} is closely related to Hofer's analysis~\cite[pp534--539]{Hofer} of
pseudo-holomorphic curves in symplectizations of contact manifolds. 
\section{Statement of Theorem~\ref{theorem}}
\label{statement}
An $m$-form $\phi$ on a Riemannian manifold $N$ is said to be of comass~$\leq1$ if
\[\phi (v_1,\dots,v_m)\leq 1\]
for every orthonormal vector fields $v_1,\dots,v_m$ on $N$.
For each $m$-form $\phi$ of comass~$\leq1$ on a Riemannian manifold $N$,
a $\phi$-submanifold $M$ of $N$ is defined to be an $m$-dimensional oriented submanifold with volume form $\phi|_M$.
Harvey and Lawson~\cite{Harvey and Lawson} proves that
$\phi$-submanifolds are minimal submanifolds of the Riemannian manifold
if $\phi$ is a closed form of comass~$\leq1$.
A closed form $\phi$ of comass~$\leq1$ is called a calibration on the Riemannian manifold.
A calibration is said to be parallel if it is a parallel differential form on the Riemannian manifold.

Let $\phi$ be a parallel calibration of degree $m>1$ on the Euclidean space $\mathbb{R}^n$.
Set
\begin{equation}
\label{psi}
\psi=\left( \partial_r \lrcorner \phi \right)|_{\mathbb{S}^{n-1}},
\end{equation}
where $\partial_r$ is the vector field in the direction of the radial coordinate $r=|\bullet|$ on $\mathbb{R}^n\setminus\{0\}$,
where $\lrcorner$ is the interior product of vector fields with differential forms,
and where $\mathbb{S}^{n-1}$ is the unit sphere in $\mathbb{R}^n$.
\begin{proposition}
\label{psi comass less than one}
$\psi$ is an $(m-1)$-form of comass~$\leq 1$ on $\mathbb{S}^{n-1}$;
in particular \[\int_X \psi \leq \Vol(X)\] for every compact $(m-1)$-dimensional oriented submanifold $X$ of $\mathbb{S}^{n-1}$.
\end{proposition}
\begin{proof}
For every orthonormal vector fields $v_1,\dots,v_{m-1}$ on $\mathbb{S}^{n-1}$,
\[ \psi(v_1,\dots,v_{m-1})=\phi(\partial_r,v_1,\dots,v_{m-1}) \leq 1\]
since $\partial_r,v_1,\dots,v_{m-1}$ are orthonormal.
\end{proof}

Let $b_0$, $b_1$ be real numbers with $b_0<b_1$.
Let $A(b_0,b_1;\mathbb{S}^{n-1})$ be the pre-image of $(b_0,b_1)\times \mathbb{S}^{n-1}$
under the polar coordinates $\mathbb{R}^n\setminus\{0\}\rightarrow (0,\infty)\times \mathbb{S}^{n-1}:a\mapsto (|a|,a/|a|)$.
Let $X$ be a compact submanifold of $\mathbb{S}^{n-1}$,
and $A(b_0,b_1;X)$ be the pre-image of $(b_0,b_1)\times X$ under the polar coordinates. 
\begin{remark}
\label{psi submanifolds}
$X$ is a $\psi$-submanifold
if and only if $A(b_0,b_1;X)$ is a $\phi$-submanifold of $A(b_0,b_1;\mathbb{S}^{n-1})$.
\end{remark}
\begin{proposition}
\label{psi minimal}
$\psi$-submanifolds are minimal submanifolds of $\mathbb{S}^{n-1}$.
\end{proposition}
\begin{proof}This follows from Remark~\ref{psi submanifolds} and the fact that $X$ is a minimal submanifold of $\mathbb{S}^{n-1}$
if and only if $A(b_0,b_1;X)$ is a minimal submanifold of $A(b_0,b_1;\mathbb{S}^{n-1})$.\end{proof}

Let $\nu$ be a normal vector field on $A(b_0,b_1;X)$ in $A(b_0,b_1;\mathbb{S}^{n-1})$.
Set
\begin{align*}
& \|\nu\|_{C^0_\mathrm{cyl}}=\sup_{A(b_0,b_1;X)} |\nu|/r ,\\
& \|\nu\|_{C^1_\mathrm{cyl}}=\sup_{A(b_0,b_1;X)} \bigl( |\nu|/r+|D\nu| \bigr),
\end{align*}
where $r$ is the radial coordinate, and $D\nu$ is the covariant derivative of $\nu$.
These are induced by the cylindrical metric $dr^2/r^2+ds^2$, where $ds^2$ is the metric on $\mathbb{S}^{n-1}$
induced by the Euclidean metric on $\mathbb{R}^n$.
Set
\begin{equation*}
\label{G}
G(\nu)=\Bigl\{ \frac{|a|}{\sqrt{|a|^2+|\nu(a)|^2}}\bigl(a+\nu(a)\bigr) \Bigm| a\in A(b_0,b_1;X) \Bigr\}.
\end{equation*}

The following theorem will be the key step to the proof of the main result;
see the proof of Proposition~\ref{uniqueness1} in Section~\ref{main result}.
\begin{theorem}
\label{theorem}
Let $m$, $n$ be integers with $1<m<n$,
let $\phi$ be a parallel calibration of degree $m$ on $\mathbb{R}^n$,
let $\psi$ be the $(m-1)$-form~\eqref{psi} on $\mathbb{S}^{n-1}$,
let $X$ be a compact $\psi$-submanifold of $\mathbb{S}^{n-1}$,
and let $\beta$ be a positive real number $<1$.
Then, there exist real numbers $\theta=\theta(m,n,X)\in (0,1/2)$, $C=C(m,n,X,\phi)>0$, $\epsilon=\epsilon(m,n,X,\phi,\beta)>0$ such that the following holds:

If $b_0$, $b_1$ are real numbers with $b_0<b_1\beta$, 
if $M$ is a closed $\phi$-submanifold of $A(b_0\beta, b_1;\mathbb{S}^{n-1})$, and if
for each $i=0,1$ there exists a normal vector field $\nu_i$ on $A(b_i\beta, b_i;X)$ in
$A(b_i \beta,b_i;\mathbb{S}^{n-1})$ such that
\begin{align*}
& M\cap A(b_i\beta, b_i;\mathbb{S}^{n-1})=G(\nu_i), \\
& \| \nu _i \|_{C^1_\mathrm{cyl}}\leq \epsilon \text{ for each } i=0,1,
\end{align*}
then there exists a normal vector field $\nu$ on $A(b_0\beta, b_1;X)$ in $A(b_0\beta, b_1;\mathbb{S}^{n-1})$ such that
\begin{align*}
& M= G(\nu) ,\\
& \nu|_{A(b_i\beta, b_i;X)}=\nu _i \text{ for each } i=0,1,\\
& \| \nu \|_{C^1_\mathrm{cyl}}\leq C\epsilon^{\theta}.
\end{align*}
\end{theorem}
Theorem~\ref{theorem} will be proved in Section~\ref{proof section}
with the help of results in Section~\ref{lemma section 1} and Section~\ref{lemma section 2}.

\section{First Lemma for Theorem~\ref{theorem}: Lemma~\ref{extension lemma}}
\label{lemma section 1}
Let $m$, $n$ be integers with $1<m<n$.
Let $r:\mathbb{R}^n\setminus\{0\}\rightarrow (0,\infty)$ be the map $a\mapsto|a|$,
and $s:\mathbb{R}^n\setminus\{0\}\rightarrow \mathbb{S}^{n-1}$ be the map $a\mapsto a/|a|$.
Let $\phi$ be a parallel calibration of degree $m$ on $\mathbb{R}^n$,
and $\psi$ be the $(m-1)$-form~\eqref{psi} on $\mathbb{S}^{n-1}$.
\begin{proposition}
\label{phi=d}
\begin{equation}
\label{phi=dr}
m\phi|_{\mathbb{R}^n\setminus\{0\}}=d(r^ms^*\psi).
\end{equation}
\end{proposition}
\begin{proof}
Since $\phi$ is an $m$-form with constant components in the Euclidean coordinates on $\mathbb{R}^n$,
the Lie derivative of $\phi$ along $r\partial_r$ is equal to $m\phi$.
Therefore, by Cartan's formula,
\begin{equation*}
\label{mpsi=d}
m\phi=d(r\partial_r\lrcorner \phi).
\end{equation*}
Therefore,
\begin{equation}
\label{psitilde}
 m\phi|_{\mathbb{R}^n\setminus\{0\}}=d(r^m\widetilde{\psi}),
\end{equation}
where $\widetilde{\psi}=r^{1-m}\partial_r\lrcorner (\phi|_{\mathbb{R}^n\setminus\{0\}})$.
Then, $\partial_r \lrcorner \widetilde{\psi}=0$,
and $\widetilde{\psi}$ is invariant under the flow generated by $r\partial_r$.
Therefore, $\widetilde{\psi}=s^*( \widetilde{\psi}|_{\mathbb{S}^{n-1}})$.
On the other hand, $\widetilde{\psi}|_{\mathbb{S}^{n-1}}=\psi$ by ~\eqref{psi}.
Thus, $\widetilde{\psi}=s^{*}\psi$.
Therefore, \eqref{psitilde} gives \eqref{phi=dr}.
\end{proof}
\begin{remark}
\label{s^*psi}
By the calculation above, $s^*\psi=r^{1-m}\partial_r\lrcorner (\phi|_{\mathbb{R}^n\setminus\{0\}})$.
\end{remark}
\begin{proposition}\label{normal to calibrated submanifold}
Let $M$ be a $\phi$-submanifold of $\mathbb{R}^n$. Then,
$(\nu\lrcorner\phi)|_M=0$ if $\nu$ is normal to $M$.
\end{proposition}
\begin{proof}
Suppose $v_1, \dots, v_{m-1}\in T_PM$ are orthonormal.
It suffices to prove:
\begin{equation}
\label{normal zero}
\phi(\nu,v_1,\dots,v_{m-1})=0.
\end{equation}
There exists $v\in T_PM$ such that $\phi(v, v_1, \dots, v_{m-1})=1$.
The function $t\mapsto \phi((\sin t)\nu+(\cos t)v,v_1,\dots,v_{m-1})$ attains maximum $1$ at $t=0$.
Differentiating it at $t=0$, therefore, gives \eqref{normal zero}.
\end{proof}
\begin{proposition}[Harvey and Lawson]
\label{dpsi}
Let $M$ be a $\phi$-submanifold of $\mathbb{R}^n$.
Then,
\begin{equation*}
\label{dpsi=}
\langle \overrightarrow{TM}, d\psi\rangle=|\partial_r \wedge \overrightarrow{TM}|^2mr^{-m},
\end{equation*}
where $\langle \bullet,\bullet\rangle$ is the pairing of $m$-vector fields and $m$-forms,
and $\overrightarrow{TM}$ is the unit simple $m$-vector field dual to the volume form of $M$.
\end{proposition}
\begin{proof}
See Harvey and Lawson~\cite[(5.13), Lemma~5.11, p65]{Harvey and Lawson}.
\end{proof}
\begin{corollary}
\label{monotonicity}
Let $M$ be a closed $\phi$-submanifold of $A(b_0,b_1;\mathbb{S}^{n-1})$.
If $b_0<c_0<c_1<b_1$, if $c_0$, $c_1$ are regular values of $r|_M:M\rightarrow I$, and if $r^{-1}(c_0)$, $r^{-1}(c_1)$ are non-empty, then 
\begin{equation}
\label{monotonicity psi}
\int_{M\cap r^{-1}(c_0)}\psi \leq\int_{M\cap r^{-1}(c_1)}\psi.
\end{equation}
\end{corollary}
\begin{proof}
By Proposition~\ref{dpsi},
\[\int_{M\cap r^{-1}( [c_0,c_1] )}d\psi\geq0.\]
This, by Stokes' theorem, gives \eqref{monotonicity psi}.
\end{proof}
The following lemma will be important in the proof of Theorem~\ref{theorem}; see Section~\ref{proof section}.
\begin{lemma}
\label{extension lemma}
Let $m$, $n$ be integers with $1<m<n$,
let $\phi$ be a parallel calibration of degree $m$ on $\mathbb{R}^n$,
let $\psi$ be the $(m-1)$-form~\eqref{psi} on $\mathbb{S}^{n-1}$,
let $X$ be a compact $\psi$-submanifold of $\mathbb{S}^{n-1}$,
let $\beta$ be a positive real numbers~$<1$,
and let $\epsilon_*$ be a positive real number.
Then, there exists a positive real number $\epsilon_{**}$ such that the following holds:

If $M$ is a closed $\phi$-submanifold of $A(\beta^2,1;\mathbb{S}^{n-1})$, if
\begin{equation}
\label{energy condition}
\int_Md\psi \leq \epsilon_{**},
\end{equation}
and if
there exists a normal vector field $\nu$ on $A(\beta,1;X)$ in $A(\beta,1;\mathbb{S}^{n-1})$ such that
\begin{align*}
&M\cap A(\beta,1;\mathbb{S}^{n-1})=G(\nu), \\
&\| \nu \|_{C^1_\mathrm{cyl}} \leq\epsilon_{**},
\end{align*}
then there exists a normal vector field $\nu'$ on $A(\beta^{3/2},1;X)$ in $A(\beta^{3/2},1;\mathbb{S}^{n-1})$ such that
\begin{align*}
&M= G(\nu'), \\
&\| \nu ' \|_{C^{1,1/2}_{\mathrm{cyl}}(\beta^{3/2} ,\beta)} \leq \epsilon_*,
\end{align*}
where $C^{1,1/2}_\mathrm{cyl}(\beta^{3/2} ,\beta)$ is the H\"older space on $A(\beta^{3/2} ,\beta;X)$
with respect to the metric $dr^2/r^2+ds^2$ on $A(\beta^{3/2} ,\beta;X)$, 
where $r$ is the radial coordinate, and $ds^2$ is the metric on $\mathbb{S}^{n-1}$ induced by the Euclidean metric
on $\mathbb{R}^n$.
\end{lemma}
\begin{proof}
Suppose there does not exist such $\epsilon_{**}$.
Then, there exists a sequence $(M_i)_{i=2,3,4,\dots}$ of closed $\phi$-submanifolds of
$A(\beta^2,1;\mathbb{S}^{n-1})$
with the following properties:
\begin{itemize}
\item[(P1)] $\int_{M_i}d\psi \leq 1/i$;
\item[(P2)] for each $i=2,3,4,\dots$, there exists a normal vector field $\nu_i$ on $A(\beta,1;X)$ in $A(\beta,1;\mathbb{S}^{n-1})$ such that
\begin{align*}
& M_i\cap A(\beta,1;\mathbb{S}^{n-1})=G(\nu_i ),\\
& \| \nu_ i \|_{C^1_\mathrm{cyl}}\leq 1/i;
\end{align*}
\item[(P3)] for each $i=1,2,\dots$, there does not exist any normal vector field $\nu_i'$ on $A(\beta^{3/2},\beta;X)$ in $A(\beta^{3/2},\beta;\mathbb{S}^{n-1})$ such that
\begin{align*}
& M_i= G(\nu_i'), \\
& \| \nu _i' \|_{C^{1,1/2}_\mathrm{cyl}}\leq \epsilon_*.
\end{align*}
\end{itemize}
By Proposition~\ref{phi=d} and (P1),
\begin{align*}
 \Vol(M_i)
&=\int_{M_i}\phi\\
&=\int_{M_i}m^{-1}d(r^ms^*\psi)\\
&=\int_{M_i}r^{m-1}dr\wedge s^*\psi+ m^{-1}r^m d\psi \\
&\leq  \left(\int_{M_i}dr\wedge s^*\psi\right)+ (mi)^{-1},
\end{align*}
whereas by Proposition~\ref{monotonicity} and (P2),
\begin{align*}
\int_{M_i}dr\wedge s^*\psi 
&=\int_{ \beta^2 }^{1} db\int_{M_i\cap r^{-1}(b)} s^*\psi\\
&\leq \int_{ \beta^2}^{1} db\limsup_{b\to \rho} \int_{M_i\cap r^{-1}(b)} s^*\psi\\
&\leq \left( 1-\beta^2 \right) \limsup_{b\to \rho} \int_{M_i\cap r^{-1}(b)} s^*\psi\\
&\leq C
\end{align*}
for some $C$ independent of $i=2,3,4,\dots$.
Therefore,
\begin{equation}
\label{boundedness}
 \sup_{i=2,3,4,\dots}\Vol(M_i)<\infty.
\end{equation}
Therefore, by Allard's compactness~\cite[Theorem 5.6]{Allard}, there exists a subsequence $(M_{i_j})_{j=2,3,4,\dots}$ such that
$(M_{i_j})_{j=2,3,4,\dots}$ converges as varifolds to some rectifiable varifold $M_{\infty}$ in $A(\beta^2,1;\mathbb{S}^{n-1})$.
Let $\|M_{\infty}\|$ denote the Radon measure on $A(\beta^2,1;\mathbb{S}^{n-1})$ induced by $M_{\infty}$.
For each $\|M_{\infty}\|$-measurable subset $E$ of $A(\beta^2,1;\mathbb{S}^{n-1})$,
let $\|M_{\infty}\|\llcorner E$ denote the restriction of $\|M_{\infty}\|$ to $E$, i.e.,
\[ \|M_{\infty}\|\llcorner E(E')=\|M_{\infty}\| (E\cap E').\]
\begin{proposition}
\label{scaling invariance}
$a^m\|M_{\infty}\|\llcorner (a^{-1}E)=\|M_{\infty}\|\llcorner E$ for every
$\|M_{\infty}\|$-measurable subset $E$ of $A(\beta^2,1;\mathbb{S}^{n-1})$ and every positive real number $a$
with $aE\subset A(\beta^2,1;\mathbb{S}^{n-1})$. 
\end{proposition}
\begin{proof}
By an approximation argument, it suffices to prove the following:
\begin{equation}
\label{approximation by smooth function}
\frac{d}{da}a^{m}\int_{(r,s)\in A(\beta^2,1;\mathbb{S}^{n-1})}f(ar)g(s) d\|M_{\infty}\| =0
\end{equation}
for every $f\in C_{\mathrm{c}}^{\infty}(\beta^2,1)$ with $a(\mathrm{supp}{f})\subset (\beta^2,1)$ and every $g\in C^{\infty}(\mathbb{S}^{n-1})$.
By (P1),
\begin{equation}
\label{11111111}
\begin{split}
&\text{the left-hand side of \eqref{approximation by smooth function}}\\
&=\frac{d}{da}\lim_{j\to \infty}
a^{m}\int_{M_{i_j}}f(ar)g \dVol(M_{i_j}) \\
&=\lim_{j\to \infty}
\int_{M_{i_j}}\frac{d}{da}((ar)^m f(ar))dr/r\wedge gs^*\psi \\
&=\lim_{j\to \infty}
\int_{M_{i_j}}a^{-1}\frac{d}{dr}((ar)^{m}f(ar)) dr\wedge gs^*\psi \\
&=\lim_{j\to \infty}
-\int_{M_{i_j}}a^{-1}(ar)^{m}f(ar) dg\wedge s^*\psi.
\end{split}
\end{equation}
By Remark~\ref{s^*psi}, Proposition~\ref{normal to calibrated submanifold}, Proposition~\ref{dpsi} and
\eqref{boundedness},
\begin{equation}
\label{22222222}
\begin{split}
&\lim_{j\to \infty}
-\int_{M_{i_j}}a^{-1}(ar)^{m}f(ar) dg\wedge s^*\psi \\
&=\lim_{j\to \infty}
\int_{M_{i_j}}a^{-1}(ar)^{m}f(ar) r^{1-m}\partial_r\lrcorner (dg\wedge \phi),\\
&=\lim_{j\to \infty}
\int_{M_{i_j}}a^{-1}(ar)^{m}f(ar) r^{1-m}\langle \mathrm{pr}_{TM_{i_j}^{\perp}}\partial_r, dg \rangle \phi \\
&\leq \lim_{j\to \infty}
a^{m-1}\sup|f|\sqrt{\int_{M_{i_j}} mr^{-m} |\mathrm{pr}_{TM_{i_j}^{\perp}}\partial_r|^2 \phi}
\sqrt{\int_{M_{i_j}}|dg|^2 \phi}\\
&\leq \lim_{j\to \infty}\mathrm{const.}\sqrt {\int_{M_{i_j}} d\psi}\\
&=0,
\end{split}
\end{equation}
where $\mathrm{pr}_{TM_{i_j}^{\perp}}$ is the projection of $\mathbb{R}^n$ onto the normal bundle
of $M_{i_j}$ in $\mathbb{R}^n$.
By \eqref{11111111} and \eqref{22222222}, the left-hand side of
\eqref{approximation by smooth function} is zero. 
This completes the proof of Proposition~\ref{scaling invariance}.
\end{proof}
By (P2), on the other hand, $\|M_{\infty}\|\llcorner A(\beta,1;\mathbb{S}^{n-1})$ is equal to
$\|A(\beta,1;X)\|$ as Radon measures on $A(\beta^2,1;\mathbb{S}^{n-1})$, where
\[\|A(\beta,1;X)\|(E)=\Vol ( A( \beta,1,;X ) \cap E).\]
Therefore, by Proposition~\ref{scaling invariance}, $M_{\infty}=A(\beta^2,1;X)$ as varifolds
in $A(\beta^2,1;\mathbb{S}^{n-1})$.
Therefore, $(M_{i_j})_{j=2,3,4,\dots}$ converges to a submanifold $A(\beta^2,1;X)$ as varifolds
in $A(\beta^2,1;\mathbb{S}^{n-1})$.
Therefore, by Allard's regularity~\cite[Theorem 8.19]{Allard}, $(M_{i_j})_{j=2,3,4,\dots}$ converges to a submanifold $A(\beta^2,1;X)$ in the local $C^{1,1/2}$-sense.
This contradicts (P3), which completes the proof of Lemma~\ref{extension lemma}.
\end{proof}
\section{Second Lemma for Theorem~\ref{theorem}: Lemma~\ref{Simon's estimate}}
\label{lemma section 2}
Let $X$ be a compact smooth Riemannian manifold,
and $V$ be a smooth real vector bundle over $X$ with a fibre metric and a metric connection. 
Let $C^{\infty}_x$ denote the space of all smooth sections of $V$ over $X$.
For every $v,v'\in C^{\infty}_x$, set
\begin{align*}
& (v,v')_{L^2_x}=\int _X \bigl( v(x),v'(x)\bigr) dx, \\
& \|v\|_{L^2_x}=\sqrt{ \int _X \bigl( v(x),v(x)\bigr) dx},
\end{align*}
where $\bigl( v(x),v'(x)\bigr)$ is the inner product on the fibre $V|_x$ over $x\in X$.
Let $F:C^{\infty}_x\rightarrow \mathbb{R}$ be a functional of the following form:
\begin{equation}
\label{multiple integral}
 Fv=\int_X f(x,v,D_x v )dx
\end{equation}
for every $v\in C^{\infty}_x$ with covariant derivative $D_xv$,
where $f=f(x,v,p)$ is a $\mathbb{R}$-valued smooth function of $x\in X$, $v\in V|_x$, and $p\in T^*_xX\otimes V|_x$.
Suppose $f$ satisfies the following conditions:
\begin{itemize}
\item[(C1)] $(v,p)\mapsto f(x,v,p)$ is a real-analytic function on the vector space $V|_x \oplus \left( T^*_xX\otimes V|_x \right)$ for every $x\in X$;
\item[(C2)] $F$ satisfies the Legendre-Hadamard condition at $0\in C^{\infty}_x$, i.e.,
\[ \frac{d^2}{dh^2}f(x,0,h^2\xi\otimes v)\Big|_{h=0}>c|\xi|^2|v|^2\]
for every $x\in X, \xi\in T_x^*X, v\in V|_x$ for some $c>0$.
\end{itemize}
Let $-\grad F:C^{\infty}_x\rightarrow C^{\infty}_x$ be the Euler--Lagrange operator of $F$, i.e.,
\[ \bigl( \grad F(v),v' \bigr)_{L^2_x}=\frac{d}{dh}F(v+hv')\Big|_{h=0}\]
for every $v,v'\in C^{\infty}_x$.
Suppose \[\grad F(0)=0,\] where $0\in C^{\infty}_x$.

Let $t_0$, $t_{\infty}$ be real numbers with $t_0<t_{\infty}$.
Consider the product $(t_0,t_{\infty})\times X$,
and the pull-back $\mathrm{pr}_2^*V$ over $(t_0,t_{\infty})\times X$,
where $\mathrm{pr}_2^*$ is the projection of $(t_0,t_{\infty})\times X$ onto $X$.
Let $C^{\infty}_{t,x}(t_0,t_{\infty})$ denote the space of
all smooth sections of $\mathrm{pr}_2^*V$ over $(t_0,t_{\infty})\times X$.
For every $u=u(t,x)\in C^{\infty}_{t,x}(t_0,t_{\infty})$,
set $u(t)=u(t,\bullet) \in C^{\infty}_x$ for each $t\in (t_0,t_{\infty})$.
For every $u=u(t,x)\in C^{\infty}_{t,x}(t_0,t_{\infty})$ and for every non-negative integers $k$, let $\|u\|_{C^{k,\mu}_{t,x}(t_0,t_{\infty})}$ be the H\"older norm with respect to the product metric $dt^2+dx^2$ on $(t_0,t_{\infty})\times X$. 
For every $u=u(t,x)\in C^{\infty}_{t,x}(t_0,t_{\infty})$, set 
\[\|u\|_{L^2_{t,x}(t_0,t_{\infty})}=\sqrt{\int_{t_0}^{t_{\infty}}\|u(t)\|_{L^2_x}^2dt}.\]
\begin{theorem}[Simon]
\label{Lojasiewicz-Simon}
There exist real numbers $\delta_0>0$, $\theta \in (0,1/2)$ depending only on $X$, $V$, $F$
such that if $t_3,t_4\in (t_0,t_{\infty})$ with $t_3<t_4$, if $u\in C^{\infty}_{t,x}(t_0,t_{\infty})$, if $\delta>0$, and if for each $t\in [t_3,t_4]$,
\begin{align*}
& \|u(t)\|_{C^{2,1/2}_x} \leq \delta_0,\\
& F(0)-F\bigl( u(t)\bigr) \leq \delta,\\
& \|\partial_t u(t)+\grad F\bigl(u(t)\bigr) \|_{L^2_x}\leq (3/4)\|\partial_t u(t)\|_{L^2_x} ,
\end{align*}
then
\begin{equation*}
\label{stability inequality}
\int_{t_3}^{t_4}\|\partial_t u(t)\|_{L^2_x}dt
\leq (4/\theta) \bigl( \bigl|  F\bigl( u(t_3)\bigr)-F(0)  \bigr|^{\theta} + \delta ^{\theta} \bigr).
\end{equation*}
\end{theorem}
\begin{proof}
The Simon condition implies the \L{}ojasiewicz inequality in the sense of Simon~\cite[Theorem~3]{Simon},
which leads to Theorem~\ref{Lojasiewicz-Simon}; see Simon~\cite[Lemma~1, p542]{Simon}.
\end{proof}
Simon~\cite{Simon} studies smooth functions $u=u(t,x)\in C^{\infty}_{t,x}(t_0,t_{\infty})$ satisfying the partial differential equation
\begin{equation}
\label{Simon's equation}
 \partial_t^2u-\partial_tu-\grad F(u)+R(u,\partial_tu,\partial_t^2u) =0,
\end{equation}
where $R:C^{\infty}_x\times C^{\infty}_x\times C^{\infty}_x\rightarrow C^{\infty}_x$ is a remainder of the following form:
\begin{equation}
\label{remainder}
\begin{split}
 R( v,v^{(1)},v^{(2)} )=& \bigl(A(x,v,D_x v,v^{(1)}) D_x^2 v \bigr)v^{(1)}\\ & +\sum_{(k,l)=(0,1),(1,1),(0,2)}B_{kl}(x,v,D_x v,v^{(1)})D_x^{l}v^{(k)}
\end{split}\end{equation}
for every $v,v^{(1)},v^{(2)}\in C^{\infty}_x$, where $A=A(x,v,p,q)$, $B_{kl}=B_{kl}(x,v,p,q)$ are smooth functions of $x\in X$, $v\in V|_x$, $p\in T^*_xX\otimes V|_x$, $q\in V|_x$ with
$A(x,v,p,q)\in \bigotimes ^2 T_xX\otimes V|_x^*$, $B_{kl}(x,v,p,q)\in \bigotimes ^l T_xX$, $B_{kl}(x,0,0,0)=0$ for every $x\in X$, for every $(k,l)=(1,0),(1,1),(2,0)$.
\begin{remark}
\label{remainder condition}
Let $R$ be of the form~\eqref{remainder} with $A$, $B$ as above.
Then, for every $C_2>0$ there exists $\delta_5=\delta_5(X,V,F,R,C_2)>0$ such that
if $\|u\|_{C^{1,1/2}_{t,x}}\leq \delta_5$, then
\begin{equation*}
\begin{split}
&|R(u(t), \partial_t u(t), \partial_t^2 u(t)) |\\
&\leq \frac{1}{16} |\partial_t u(t)| + \frac{1}{16}  |\partial_tu(t)| +\frac{1}{16C_2} |D_x\partial_t u(t)|+ \frac{1}{8} |\partial_t^2 u(t)|.
\end{split}
\end{equation*}
\end{remark}
\begin{remark}
\label{minimal submanifolds equation}
Suppose $X$ is a minimal submanifold of $\mathbb{S}^{n-1}$,
and $V$ is the normal bundle of $X$ in $\mathbb{S}^{n-1}$.
For each normal vector field $v$ on $X$ in $\mathbb{S}^{n-1}$, set
\[ G(v)=\Bigl\{ \frac{|x|}{\sqrt{|x|^2+|v(x)|^2}}\bigl(x+v(x)\bigr) \Bigm| x\in X \Bigr\},\]
which is a compact submanifold of $\mathbb{S}^{n-1}$
if $\sup_{x\in X}|v(x)|\leq \epsilon_0$ for some $\epsilon_0>0$ depending only on $X\subset \mathbb{S}^{n-1}$.
Let $C^{\infty}_x$ be the space of all normal vector fields $v$ on $X$ in $\mathbb{S}^{n-1}$ with $\sup_{x\in X}|v(x)|\leq \epsilon_0$.
Let $F:C^{\infty}_x\rightarrow \mathbb{R}$ be the following functional:
\begin{equation}
\label{F}
F(v)=\Vol \left(G(v)\right).
\end{equation}
$F$ is of the form~\eqref{multiple integral} with integrand $f$ satisfying (C1), (C2), and $\grad F(0)=0$.
Set $b_0=e^{-t_{\infty}/m}$, $b_1=e^{-t_0/m}$; $b_0<b_1$.
Let $A(b_0,b_1;\mathbb{S}^{n-1})$, $A(b_0,b_1;X)$ be as in Section~\ref{statement}.
Let $\nu$ be a normal vector field on $A(b_0,b_1;X)$ in $A(b_0,b_1;\mathbb{S}^{n-1})$.
Set $u(t,x)=e^{t/m}\nu(e^{-t/m}x)$; $u\in C^{\infty}_{t,x}(t_0,t_{\infty})$.
Let $G(\nu)$ be as in Section~\ref{statement}.
If $G(\nu)$ is a minimal submanifold,
then $u$ satisfies Simon's equation~\eqref{Simon's equation},
where $F$ is the functional~\eqref{F}, and $R$ is some remainder of the form~\eqref{remainder} depending only on $m$, $n$, $X$. 
\end{remark}
Let $H:C^{\infty}_x\rightarrow C^{\infty}_x$ be the linearized operator of $\grad F$ at $0\in C^{\infty}_x$.
Simon's equation~\eqref{Simon's equation} is of the following form:
\begin{equation*}
\partial_t^2u-\partial_tu-Hu=\sum_{0\leq k+l\leq 2}E_{kl}(x,u,D_xu,\partial_tu)D_x^l\partial_t^ku,
\end{equation*}
where $E_{kl}=E_{kl}(x,u,p,q)$ are smooth functions of $x\in X$, $v\in V|_x$, $p\in T^*_xX\otimes V|_x$, $q\in V|_x$ with
$E_{kl}(x,v,p,q)\in \bigotimes ^l T_xX$, $E_{kl}(x,0,0,0)=0$
for every $x\in X$, for every $k,l$ with $0\leq k+l\leq2$.
\begin{remark}
\label{Schauder?}
Let $t_1$, $t_6$ be real numbers with $t_1<t_6$.
There exists $\delta_2=\delta_2(X,V,F,R)>0$ such that if $u\in C^{\infty}_{t,x}$
with $\|u\|_{C^{1,1/2}_{t,x}(t_1,t_6)}\leq \delta_2$, then
\[\max_{0\leq k+l\leq2}\|E_{kl}\|_{C^{0,1/2}_{t,x}(t_1,t_6)}\leq \delta_1,\]
where $\delta_1=\delta_1(X,V,F)$ is as in Remark~\ref{Schauder} below.
\end{remark}
\begin{remark}
\label{Schauder}
Let $T>0$.
The Legendre-Hadamard condition~(C2) implies that 
$\partial_t^2-\partial_t-H$ is uniformly elliptic on $C^{\infty}_{t,x}(t_0,t_{\infty})$.
Therefore, there exists $\delta_1=\delta_1(X,V,F)>0$ be such that
if $w \in C^{\infty}_{t,x}(-T/3,T/3)$, if $a_{kl}(t,x)\in C^{\infty}_{t,x}(-T/3,T/3)$, and if
\begin{align*}
&\partial_t^2w-\partial_tw-Hw=\sum_{0\leq k+l\leq2}{a}_{kl}(t,x) D_x^l\partial_t^k w,\\
&\max_{0\leq k+l\leq2}\|a_{kl}\|_{C^{0,1/2}_{t,x}(-T/3,T/3)}\leq \delta_1,
\end{align*}
then,
\begin{equation}
\|w\|_{C^{2,1/2}_{t,x}(-T/5,T/5)}\leq C_1\|w\|_{L^2_{t,x}(-T/4,T/4)}
\end{equation}
for some $C_1=C_1(X,V,F,T)>0$.
This is a Schauder estimate; see Morrey~\cite[Chapter~6]{Morrey}.
\end{remark}
\begin{theorem}[Simon]
\label{growth theorem}
There exist $h>0$, $T>0$, $\delta_3 >0$ depending only on $X$, $V$, $F$ such that for any $j\in\{3,4,5,\dots\}$ the following holds:

If $w \in C^{\infty}_{t,x}(0,jT)$, $a_{kl}(t,x)\in C^{\infty}_{t,x}(0,jT)$, and
\begin{align*}
&\partial_t^2w-\partial_tw-Hw=\sum_{0\leq k+l\leq2}a_{kl}(t,x) D_x^l\partial_t^k w,\\
&\max_{0\leq k+l\leq2}\sup_{(t,x)\in (0,jT)\times X} \left| {a}_{kl}(t,x) \right| \leq \delta_3,\\
&\|w\|_{L^2_{t,x}(0,jT)}<\infty,
\end{align*}
then, there exist integers $i_1$, $i_2$ with $0\leq i_1\leq i_2 \leq j$ such that:
if $1<i_1$, then
\begin{equation*}
\|w\|_{L^2_{t,x}(iT,(i+1)T)} \leq e^{-h T} \|w\|_{L^2_{t,x} ( (i-1)T,iT)}
\end{equation*}
for each $i\in \{1,\dots,i_1-1\}$; if $i_1<i_2$, then
\begin{equation*}
\| w(t) \|_{L^2_x}\leq (3/2) \| w(t') \|_{L^2_x}
\end{equation*}
for $t,t'\in(i_1T,i_2T)$ with $|t'-t|\leq T$, and
\begin{equation*}
\| \partial_tw(t) \|_{L^2_x}\leq (1/2) \| w(t) \|_{L^2_x}
\end{equation*}
for each $t\in (i_1T,i_2T)$; if $i_2<j-1$, then
\begin{equation*}
\|w\|_{L^2_{t,x} ( (i-1)T,iT) } \leq e^{-h T}\|w\|_{L^2_{t,x} (iT,(i+1)T)}
\end{equation*}
for each $i\in \{i_2+1,\dots,j-1\}$.
\end{theorem}
\begin{proof}
See Simon~\cite[Theorem~3.4]{Simon2}.
\end{proof}
The following lemma will be important in the proof of Theorem~\ref{theorem}; see Section~\ref{proof section}.
\begin{lemma}
\label{Simon's estimate}
Let $X$ be a compact smooth Riemannian manifold,
$V$ be a smooth real vector bundle over $X$ with a fibre metric and a metric connection,
$F:C^{\infty}_x\rightarrow \mathbb{R}$ be a functional with $\grad F(0)=0$ and of the form~\eqref{multiple integral} with integrand $f$ satisfying \upshape{(C1)}, \upshape{(C2)}.
Then, there exist real numbers $C>0$ and $\theta\in (0,1/2)$ such that
for every remainder $R$ of the form~\eqref{remainder},
there exists $\delta_*>0$ such that the following holds:

If $t_0< t_*$, if $0<\delta<1$,
if $u \in C^{\infty}_{t,x}(t_0,t_*)$ satisfies Simon's equation~\eqref{Simon's equation},
and if
\begin{align}
\label{solution with 11}
& \|u\|_{C^{1,1/2}_{t,x}(t_0,t_*)}\leq \delta_*, \\
\label{solution with 22}
& \limsup_{t \to t_0} \|u(t)\|_{L^2_x}\leq \delta, \\
\label{solution with 33}
& \sup_{t\in (t_0,t_*)} \Bigl( {F}(0)-{F}\bigl( u(t) \bigr) \Bigr) \leq \delta,\\
\label{solution with 44}
& \|\partial_tu\|_{L^2_{t,x}(t_0,t_*)}\leq \sqrt{\delta},
\end{align}
then
\begin{equation*}
\sup_{t\in(t_0,t_*)} \| u(t) \|_{L^2_x}<C_*\delta^{\theta}.
\end{equation*}
\end{lemma}
\begin{proof}
By \eqref{solution with 22}, it suffices to prove the following:
\begin{equation}
\label{a-priori estimate}
\int_{t_0}^{t_*} \| \partial_tu(t) \|_{L^2_x}dt<C_*\delta^{\theta}.
\end{equation}

Let $T=T(X,V,F)>0$ be as in Theorem~\ref{growth theorem}.
If $t_*-t_0<5T$, then
\begin{equation}
\label{the first}
\begin{split}
\int_{t_0}^{t_*} \| \partial_tu(t) \|_{L^2_x}dt
&\leq \sqrt{t_*-t_0}\|\partial_t u\|_{L^2_{t,x}(t_0,t_*)}\\
&\leq \sqrt{5T}\delta,
\end{split}
\end{equation}
which gives the conclusion~\eqref{a-priori estimate}.

Suppose $t_*-t_0\geq 5T$.
Let $t_1, t_6\in(t_0,t_*)$ be such that $T/2\leq t_1-t_0\leq T$, $T/2\leq t_*-t_6\leq T$ and 
$t_6-t_1=jT$
for some  $j\in \{3,4,5,\dots \}$.
Then,
\begin{align}
\label{t_0 to t_1} &\int_{t_0}^{t_1}\|\partial_t u(t)\|_{L^2_{t,x}}\leq \sqrt{T}\delta .\\
\label{t_6 to t_*} &\int_{t_6}^{t_*}\|\partial_t u(t)\|_{L^2_{t,x}}\leq \sqrt{T}\delta .
\end{align}
in the same way as \eqref{the first}.

Suppose $\delta_*<\delta_2$, where $\delta_2=\delta_2(X,V,F,R)>0$ as in Remark~$\ref{Schauder?}$.
By \eqref{solution with 11}, Remark~\ref{Schauder?} and Remark~\ref{Schauder},
\[\|u\|_{C^{2,1/2}_{t,x}(t_1,t_6)}\leq C_1\delta_*,\]
where $C_1=C_1(X,V,F;T)>0$ be as in Remark~\ref{Schauder}, so that
$C_1=C_1(X,V,F)$ since $T=T(X,V,F)>0$.
It suffices therefore to prove \eqref{a-priori estimate}
under the assumption that 
\[ \|u\|_{C^{2,1/2}_{t,x}(t_1,t_6)}\leq \delta_{**}\]
for some $\delta_{**}=\delta_{**}(X,V,F,R)>0$.

Set $u'=\partial_tu$.
Then,
\begin{equation}
\label{linearized equation'}
\partial_t^2u'-\partial_tu'-Hu'=\sum_{0\leq k+l\leq2}\tilde{a}_{kl}(t,x) D_x^l\partial_t^k u'
\end{equation}
for some smooth functions $\tilde{a}_{kl}(t,x)$ of $t\in (t_0,t_*)$, $x\in X$ with $\tilde{a}_{kl}(t,x)\in \bigotimes^l T_xX$ for every $k$, $l$ with $0\leq k+l\leq2$.
By an argument similar to Remark~\ref{Schauder?}, there exists $\delta_4=\delta_4(X,V,F,R)>0$
such that if $\|u\|_{C^{2,1/2}_{t,x}(t_1,t_6)}\leq \delta_4$, then
\begin{align}
\label{error1}
&\max_{0\leq k+l\leq2}\|E_{kl}\|_{C^{0,1/2}_{t,x}(t_1,t_6)}\leq \delta_1,\\
\label{error2}
&\max_{0\leq k+l\leq2}\|\tilde{a}_{kl}\|_{C^{0,1/2}_{t,x}(t_1,t_6)}\leq \min\{\delta_1,\delta_3\},
\end{align}
where $\delta_1=\delta_1(X,V,F)$ as in Remark~\ref{Schauder}, and $\delta_3=\delta_3(X,V,F)$ as in Theorem~\ref{growth theorem}.

Suppose $\|u\|_{C^{2,1/2}_{t,x}(t_1,t_6)}\leq \delta_4$, where $\delta_4=\delta_4(X,V,F,R)>0$ as above.
By \eqref{linearized equation'} and \eqref{error2}, $u'$ satisfies the assumption in Theorem~\ref{growth theorem}.
Therefore, there exists integers $i_1$, $i_2$ with $0\leq i_1\leq i_2 \leq j$ such that: if $1<i_1$, then
\begin{equation}
\label{growth1}
\|\partial_tu\|_{L^2_{t,x}(t_1+iT,t_1+(i+1)T)} \leq e^{-h T} \|\partial_tu\|_{L^2_{t,x} (t_1+(i-1)T,t_1+iT)}
\end{equation}
for each $i\in \{1,\dots,i_1-1\}$; if $i_1<i_2$, then
\begin{equation}
\label{growth2}
\| \partial_tu(t) \|_{L^2_x}\leq (3/2) \| \partial_tu(t') \|_{L^2_x}
\end{equation}
for $t,t'\in(t_1+i_1T,t_1+i_2T)$ with $|t'-t|\leq T$, and
\begin{equation}
\label{growth3}
\| \partial_t^2 u(t) \|_{L^2_x}\leq (1/2) \| \partial_t u(t) \|_{L^2_x}
\end{equation}
for each $t\in (t_1+i_1T,t_1+i_2T)$; if $i_2<j-1$, then
\begin{equation}
\label{growth4}
\|\partial_t u\|_{L^2_{t,x} ( t_1+(i-1)T,t_1+iT) } \leq e^{-h T}\|\partial_t u\|_{L^2_{t,x} (t_1+iT,t_1+(i+1)T)}
\end{equation}
for each $i\in \{i_2+1,\dots,j-1\}$; here $h=h(X,V,F)>0$ as in Theorem~\ref{growth theorem}.
By the method in \eqref{the first}, and by \eqref{growth1}, 
\begin{equation}
\label{t_1 to t_2}
\begin{split}
 \int_{t_1}^{t_1+i_1T}\|\partial_t u(t)\|_{L^2_x}dt
& \leq \sum_{i=1}^{i_1-1}\int_{t_1+iT}^{t_1+(i+1)T}\|\partial_t u(t)\|_{L^2_x}dt\\
&\leq \sum_{i=0}^{i_1-1}\sqrt{T}\|\partial_t u\|_{L^2_{t,x} ( t_1+iT,t_1+(i+1)T)}\\
&\leq \sum_{i=0}^{i_1-1}\sqrt{T}e^{-ihT}\|\partial_t u\|_{L^2_{t,x} ( t_1,t_1+T)}\\
&\leq \sqrt{T}(1-e^{-hT})^{-1}\|\partial_t u\|_{L^2_{t,x} ( t_1,t_1+T)}\\
&\leq \sqrt{T}(1-e^{-hT})^{-1}\delta.
\end{split}
\end{equation}
Set $t_5=t_1+i_2T$.
As \eqref{growth1} gives \eqref{t_1 to t_2}, so \eqref{growth4} gives the following:
\begin{equation}
\label{t_5 to t_*}
\begin{split}
 \int_{t_5}^{t_6}\|\partial_t u(t)\|_{L^2_x}dt\leq \sqrt{T}(1-e^{-hT})^{-1}\delta.
\end{split}
\end{equation}
If $i_1=i_2$, then \eqref{t_0 to t_1}, \eqref{t_1 to t_2} and \eqref{t_5 to t_*} imply the conclusion~\eqref{a-priori estimate}.

Suppose $i_1<i_2$.
Set $t_2=t_1+i_1T$, $t_3=t_2+T/3$, $t_4=t_5-T/3$.
Then,
\begin{align}
\label{t_2 to t_3}
& \int_{t_2}^{t_3}\|\partial_t u(t)\|_{L^2_x}dt\leq \sqrt{T/3}\delta, \\
\label{t_2 to t_3'}
& \int_{t_2}^{t_3+T/4}\|\partial_t u(t)\|_{L^2_x}dt\leq \sqrt{7T/12}\delta, \\
\label{t_4 to t_5}
& \int_{t_4}^{t_5}\|\partial_t u(t)\|_{L^2_x}dt\leq \sqrt{T/3}\delta.
\end{align}
in the same way as \eqref{the first}.
 
For each $t\in[t_3,t_4]$, on the other hand, by \eqref{error1}, \eqref{error2} and Remark~\ref{Schauder},
\begin{align}
\label{Schauder1}
& \|\partial_t u\|_{C^{2,1/2}_{t,x} (t-T/5,t+T/5)}\leq C_1\|\partial_t u\|_{L^{2}_{t,x} (t-T/4,t+T/4)},\\
\label{Schauder2}
& \|u\|_{C^{2,1/2}_{t,x} (t-T/5,t+T/5)}\leq C_1\|u\|_{L^{2}_{t,x} (t-T/4,t+T/4)},
\end{align}
where $C_1=C_1(X,V,F;T)>0$ as in Remark~\ref{Schauder}, so that $C_1=C_1(X,V,F)$ since $T=T(X,V,F)$.
By \eqref{Schauder1} and \eqref{growth2}, for each $t\in[t_3,t_4]$,
\begin{equation}
\label{k>0}
\begin{split}
\|D_x\partial_t  u(t)\|_{L^2_x} 
&\leq \sqrt{\Vol(X)} \sup_{x\in X}|D_x \partial_t u(t,x)| \\
&\leq \sqrt{\Vol(X)}\|\partial_t u\|_{C^{2,1/2}_{t,x} (t-T/5,t+T/5)}\\
&\leq \sqrt{\Vol(X)}C_1\|\partial_t u\|_{L^{2}_{t,x} (t-T/4,t+T/4)}\\
&\leq \sqrt{\Vol(X)}C_1(3/2)\sqrt{T/2}\|\partial_t u(t)\|_{L^2_x}\\
&\leq C_2\|\partial_t u(t)\|_{L^2_x},
\end{split}
\end{equation}
where $C_2=\sqrt{\Vol(X)}C_1(3/2)\sqrt{T/2}$, so that $C_2=C_2(X,V,F)>0$.
Therefore, by Remark~\ref{remainder condition}, \eqref{growth3} and Simon's equation~\eqref{Simon's equation}, for each $t\in[t_3,t_4]$,
\begin{equation}
\label{equation234}
\begin{split}
\| \partial_t u(t)+\grad F\bigl(u(t)\bigr) \|_{L^2_x}
&\leq \| \partial_t^2u(t)+R\left( u(t),\partial_t u(t),\partial_t^2u(t) \right)\|_{L^2_x}\\
&\leq \left( \frac{1}{2}+\frac{1}{16}+\frac{1}{16}+\frac{1}{16C_2}C_2+\frac{1}{8}\frac{1}{2}  \right)\|\partial_t u(t)\|_{L^2_x} \\
&\leq (3/4)\| \partial_t u(t) \|_{L^2_x}.
\end{split}
\end{equation}
if $\|u\|_{C^{2,1/2}_{t,x}(t_1,t_6)}\leq \delta_5$, where $\delta_5=\delta_5(X,V,F,R,C_2)>0$
as in Remark~\ref{remainder condition}, so that $\delta_5=\delta_5(X,V,F,R)$.

Suppose $\|u\|_{C^{2,1/2}_{t,x}(t_1,t_6)} \leq \min\{\delta_5,\delta_0\}$, where $\delta_5=\delta_5(X,V,F,R)$
as above, and $\delta_0=\delta_0(X,V,F)>0$ as in Theorem~\ref{Lojasiewicz-Simon}.
By Theorem~\ref{Lojasiewicz-Simon}, \eqref{equation234} and \eqref{solution with 33},
\begin{equation*}
\int_{t_3}^{t_4}\|\partial_t u(t)\|_{L^2_x} dt \leq (4/\theta) \left( \Bigl|F\bigl( u(t_3)\bigr)-F(0)\Bigr|^{\theta}+\delta^{\theta} \right),
\end{equation*}
where $\theta=\theta(X,V,F)>0$ as in Theorem~\ref{Lojasiewicz-Simon}.
Since $\grad F(0)=0$,
there exists $C_0=C_0(X,V,F)>0$ such that if $v\in C^{\infty}_x$, $\|v\|_{C^{2,1/2}_{x}}<\delta_0$, then
\[ |F(v)-F(0)| \leq C_0\|v\|_{C^1_x}^2.\]
Therefore, by \eqref{Schauder2},
\begin{equation*}
\begin{split}
\bigl|F\bigl(u(t_3)\bigr)-F(0)\bigr|
&\leq C_0\|u(t_3)\|_{C^1_x}^2\\
&\leq C_0C_1\|u\|_{L^2_{t,x} (t_3-T/4,t_3+T/4)}^2\\
&\leq C_0C_1(T/2)\sup_{t\in (t_3-T/4,t_3+T/4)}\|u(t)\|_{L^2_x}^2,
\end{split}
\end{equation*}
whereas by \eqref{solution with 22}, \eqref{t_0 to t_1}, \eqref{t_1 to t_2} and \eqref{t_2 to t_3'},
\begin{equation*}
\begin{split}
\sup_{t\in (t_3-T/4,t_3+T/4)}\|u(t)\|_{L^2_x}
&\leq \limsup_{t\to t_0}\|u(t)\|_{L^2_x}+ \int_{t_0}^{t_3+T/4}\|\partial_t u(t)\|_{L^2_x}dt\\
&\leq \left( 1+\sqrt{T}+ \sqrt{T}(1-e^{-hT})^{-1} +\sqrt{7T/12} \right)\delta ,
\end{split}
\end{equation*}
Thus,
\begin{equation}
\label{t_3 to t_4}
 \int_{t_3}^{t_4}\|\partial_tu(t)\|_{L^2_x}dt\leq C_*\delta^{\theta}
\end{equation}
for some constant $C_*$ depending only on $X$, $V$, $F$.
Thus, \eqref{t_0 to t_1}, \eqref{t_1 to t_2}, \eqref{t_2 to t_3}, \eqref{t_3 to t_4}, \eqref{t_4 to t_5}, \eqref{t_5 to t_*} and \eqref{t_6 to t_*}
imply the conclusion~\eqref{a-priori estimate}.

The proof of Lemma~\ref{Simon's estimate} is thus complete.
\end{proof}
\section{Proof of Theorem~\ref{theorem}}
\label{proof section}
We now prove Theorem~\ref{theorem}.

Let $\phi$ be a parallel calibration of degree $m>1$ on the Euclidean space $\mathbb{R}^n$.
Let $\psi$ be the $(m-1)$-form
$( \partial_r \lrcorner \phi )|_{\mathbb{S}^{n-1}}$ on the unit sphere $\mathbb{S}^{n-1}$.
Let $X$ be a compact $\psi$-submanifold of $\mathbb{S}^{n-1}$.

By Proposition~\ref{psi minimal}, $X$ is a minimal submanifold of $\mathbb{S}^{n-1}$.
Let $F$, $R$ be as in Remark~\ref{minimal submanifolds equation}.
Let $\theta\in (0,1/2)$, $C_*>0$, $\delta_*>0$ be as in Lemma~\ref{Simon's estimate},
so that $\theta=\theta(m,n,X)$, $C_*(m,n,X)>0$, $\delta_*(m,n,X)>0$.

Let $\beta$ be a positive real number~$<1$.
Suppose $\epsilon_*>0$ is so small that if $\|\nu\|_{C^0_{\mathrm{cyl}}}\leq \epsilon_*$, then:
\begin{equation}
\label{graph}
\begin{split}
&G(\nu)\text{ is a closed submanifold of }A(b_0,b_1;\mathbb{S}^{n-1}),\\
&\text{and if }G(\nu)=G(\nu')\text{ for some }\nu'\text{, then }\nu=\nu'
\end{split}
\end{equation}
in the notation of Section~\ref{statement}.
Let $\epsilon_{**}>0$ be as in Lemma~\ref{extension lemma}, so that
$\epsilon_{**}=\epsilon_{**}(m,n,X,\phi,\beta,\epsilon_*)=$.

Suppose now that $\epsilon >0$, that $b_0>0$, $b_1>0$ with $b_0<b_1\beta$,
that $M$ is a closed $\phi$-submanifold of $A(b_0\beta,b_1;\mathbb{S}^{n-1})$, and
that for each $i=0,1$ there exists a normal vector field $\nu_i$ on $A(b_i\beta,b_i;X)$ in $A(b_i\beta,b_i;\mathbb{S}^{n-1})$ such that
\begin{equation}
\label{nu_i}
\begin{split}
& M\cap A(b_i\beta,b_i;\mathbb{S}^{n-1})=G(\nu_i), \\
& \| \nu _i \|_{C^1_r}\leq \epsilon \text{ for each } i=0,1.
\end{split}
\end{equation}

For each $i=0,1$, by $\|\nu_i\|_{C^1_r}\leq \epsilon$,
\begin{equation}
\label{c_0}
\sup_{b\in (b_i\beta,b_i)}\left| \int_X \psi -\int_{M\cap r^{-1}(b)}s^*\psi \right| \leq c_0\epsilon.
\end{equation}
for some $c_0=c_0(m,n,X,\phi)>1$.
Therefore, by Stokes's theorem,
\begin{equation}
\label{energy estimate}
\int_{M} s^*d\psi \leq 2 c_0\epsilon.
\end{equation}
Therefore, if $\epsilon<\epsilon_{**}$, then by \eqref{nu_i} with $i=1$, and by Lemma~\ref{extension lemma},
there exists a normal vector field $\nu$ on $A(b_1\beta^{3/2},b_1;X)$ in $A(b_1\beta^{3/2},b_1;\mathbb{S}^{n-1})$
such that
\begin{equation}
\label{non-empty}
\begin{split}
& M\cap A(b_1\beta^{3/2},b_1;\mathbb{S}^{n-1})=G(\nu),\\
& \|\nu\|_{C^{1,1/2}_{\mathrm{cyl}}(b_1\beta^{3/2},b_1)} \leq \epsilon_*.
\end{split}
\end{equation}

Suppose $\epsilon<\epsilon_{**}$.
Let $S_*$ be the set of all $b_*\in [b_0,b_1\beta)$ such that
there exists a normal vector field $\nu$ on $A(b_*,b_1;X)$ in $A(b_*,b_1;\mathbb{S}^{n-1})$
such that
\begin{equation}
\label{nu}
\begin{split}
& M\cap A(b_*,b_1;\mathbb{S}^{n-1})=G(\nu),\\
& \|\nu\|_{C^{1,1/2}_{\mathrm{cyl}}(b_*,b_1)} \leq \epsilon_*.
\end{split}
\end{equation}
$S_*$ is non-empty since $b_1\beta^{3/2}\in S_*$ by \eqref{non-empty}.
\begin{proposition}
\label{S_*}
Suppose $b_*\in S_*\cap [b_0,b_1)$,
and let $\nu$ be as in \eqref{nu}.
Set $u(t,x)=e^{t/m}\nu(e^{-t/m}x)$, $t_0=-\log{b_1}/m$ and $t_*=-\log{b_*}/m$.
Then,
\[ \|u\|_{L^{2}_{t,x}(t_0,t_*)}\leq C_*(4c_0\epsilon)^{\theta}\]
if $\epsilon<(4c_0)^{-1}$ and $\epsilon_*<\epsilon_1$ for some $\epsilon_1=\epsilon_1(m,n,X)>0$.
\end{proposition} 
\begin{proof}
By Remark~\ref{minimal submanifolds equation},
$u$ satisfies Simon's equation~\eqref{Simon's equation} with respect to $F$ and $R$.
By Proposition~\ref{psi comass less than one}, Corollary~\ref{monotonicity} and \eqref{c_0},
\begin{equation}
\label{solution with 33'}
\begin{split}
&\sup_{b\in (b_*,b_1)} \Vol(X) -\Vol \left( s\left( M\cap r^{-1}(b) \right) \right) \\
&\leq \sup_{b\in (b_*,b_1)} \int_X \psi -\int_{M\cap r^{-1}(b)}s^*\psi \\
&\leq \int_X \psi -\int_{ M\cap r^{-1}(b_0)} s^*\psi \\
&\leq c_0\epsilon.
\end{split}
\end{equation}
By calculation, if $\|\nu\|_{C^1_{\mathrm{cyl}}}\leq \epsilon_1$ for some $\epsilon_1=\epsilon_1(m,n,X)>0$, then
\begin{equation}
\label{L^2-estimate}
|r\partial_r\nu|^2 \leq 2 |\partial_r\wedge \overrightarrow{TM}|^2,
\end{equation}
where $M=G(\nu)$, and $\overrightarrow{TM}$ is as in Proposition~\ref{dpsi}.
Suppose $\epsilon_*< \epsilon_1 $.
Then, by Proposition~\ref{dpsi}, \eqref{energy estimate} and \eqref{L^2-estimate}, 
\begin{equation}
\label{solution with 44'}
\begin{split}
\int_{M} |r\partial_r\nu|^2 /r^m \dVol(M)
&\leq 2\int_{M} |\partial_r\wedge \overrightarrow{TM}|^2/r^m \dVol(M) \\
&= 2\int_{M} s^*d\psi \\
&\leq 4c_0\epsilon.
\end{split}
\end{equation}
Now, \eqref{nu}, \eqref{nu_i} with $i=1$, \eqref{solution with 33'}, \eqref{solution with 44'}
imply \eqref{solution with 11}, \eqref{solution with 22}, \eqref{solution with 33}, \eqref{solution with 44}
in Lemma~\ref{Simon's estimate},
respectively.
Therefore, by Lemma~\ref{Simon's estimate},
\[\sup_{t\in(t_0,t_*)} \| u(t) \|_{L^2_x}<C_*(4c_0\epsilon)^{\theta} \]
if $\epsilon<(4c_0)^{-1}$.
\end{proof}
Suppose $\epsilon<(4c_0)^{-1}$ and $\epsilon_*<\epsilon_1$.
Suppose $b_*\in S_*$.
Set $\tau=\log{\beta}/m$.
Then, by interpolation and Proposition~\ref{S_*},
\begin{equation}
\begin{split}
\label{interpolation}
\|\nu\|_{C^1_{\mathrm{cyl}}(b_*,b_*/\beta)}&=\| u \|_{C^1_{t,x}(t_*-\tau,t_*)}\\
&\leq \epsilon_{**}(2\epsilon_*)^{-1}\|u\|_{C^{1,1/2}_{t,x}(t_*-\tau,t_*)}+c_1\sup_{t\in (t_*-\tau,t_*)}\|u(t)\|_{L^2_x}\\
&\leq \epsilon_{**} / 2+ c_1 C_*( 4 c_0  \epsilon)^{\theta}\\
&\leq \epsilon_{**}
\end{split}
\end{equation}
if $\epsilon<\epsilon_{**}(2c_1C_*)^{-1}(4c_0\epsilon)^{-\theta}$,
where $c_1$=$c_1(m,n,X,\tau,\epsilon_{**}(2\epsilon_*)^{-1})>0$.

Suppose $\epsilon<\epsilon_{**}(2c_1C_*)^{-1}(4c_0\epsilon)^{-\theta}$.
By \eqref{interpolation}, \eqref{energy estimate} and Lemma~\ref{extension lemma},
$b_*\beta^{1/2}\in S_*$.
Therefore, $b_*$ is an interior point in $S_*$.
Thus, $S_*$ is an open subset of $[b_0,b_1\beta)$.

By \eqref{graph}, $S_*$ is a closed subset of $[b_0,b_1\beta)$.
Thus, $S_*$ is a non-empty open closed subset of $[b_0,b_1\beta)$.
Therefore, $S_*=[b_0,b_1\beta)$.
In particular, $b_0\in S_*$.
Set $t_{\infty}=-\log{b_0}/m$. Then, by \eqref{nu_i} with $i=0$ and Proposition~\ref{S_*},
\[ \sup_{t\in (t_0,t_{\infty}+\tau)}\|u(t)\|_{L^{2}_x}\leq C_*(4c_0\epsilon)^{\theta}.\]
Therefore, by the Schauder estimate in Remark~\ref{Schauder},
\[ \|u \|_{C^{1}_{t,x}(t_0,t_{\infty}+\tau)}\leq C\epsilon^{\theta}\]
for some $C=C(C_*,c_0)=C(m,n,X,\phi)>0$.
This completes the proof of Theorem~\ref{theorem}.
\section{The Main Result}
\label{main result}
Let $(x^1,y^1,x^2,y^2,x^3,y^3)$ be the coordinates on $\mathbb{R}^6$,
and $\omega_0$ the symplectic form $dx^1\wedge dy^1+dx^2\wedge dy^2+dx^3\wedge dy^3$ on $\mathbb{R}^6$.
Let $J_0$ be the complex structure on $\mathbb{R}^6$ which maps $\partial/\partial x^{\alpha}$ to $\partial/\partial y^{\alpha}$ for every $\alpha=1,2,3$,
and $\Omega_0$ the complex volume form $dz^1\wedge dz^2 \wedge dz^3$ on $\mathbb{R}^6$,
where $z^{\alpha}=x^{\alpha}+iy^{\alpha}$ for every $\alpha=1,2,3$.
Harvey and Lawson~\cite{Harvey and Lawson} prove that $\re{\Omega_0}$ is a calibration with respect to the metric $g_0=\sum_{\alpha=1}^3dz^{\alpha}d\overline{z^{\alpha}}$ on $\mathbb{R}^6$.
$\re{\Omega_0}$-submanifolds of $(\mathbb{R}^6,g_0)$ are called special Lagrangian submanifolds of
$(\mathbb{R}^6,\omega_0,J_0,\Omega_0)$;
this is well-defined since $g_0=\omega_0(\bullet,J_0\bullet)$.

Let $\mathbb{R}^3_1$, $\mathbb{R}^3_2$ be special Lagrangian planes in $(\mathbb{R}^6,\omega_0,J_0,\Omega_0)$ such that $\mathbb{R}^3_1 \oplus \mathbb{R}^3_2 =\mathbb{R}^6$.
Lawlor~\cite{Lawlor} gives an explicit construction of $L$, $f$ with the following properties:
\begin{itemize}
\item [(L1)]$L$ is a closed special Lagrangian submanifold of $(\mathbb{R}^6,\omega_0,J_0,\Omega_0)$;
\item [(L2)]$f$ is a diffeomorphism of $\mathbb{R}\times \mathbb{S}^2$ into $\mathbb{R}^6$;
\item [(L3)]$L$ is the image of $f:\mathbb{R}\times \mathbb{S}^2 \rightarrow \mathbb{R}^6$;
\item [(L4)]there exists $R>0$ such that
\[ |f_R-i_R|/r+|d(f_R-i_R)|=O(r^{-3}) \text{ with respect to the metric }g_0\text{ on }\mathbb{R}^6,\]
where $r$ is the projection of $\mathbb{R}\times \mathbb{S}^2$ onto $\mathbb{R}$, $f_R$ is the restriction of $f$ to $(\mathbb{R}\setminus[-R,R])\times \mathbb{S}^2$,
and $i_R$ is the inclusion of $(\mathbb{R}\setminus[-R,R])\times \mathbb{S}^2$ into $\mathbb{R}^6$ whose image
is $\mathbb{R}^3_1\cup \mathbb{R}^3_2\setminus B(R)$. 
\end{itemize}
Here, $B(R)$ is the ball of radius $R$ centred at $0$ in $(\mathbb{R}^6,g_0)$.

Let $\mathbb{T}^6$ be the torus $\mathbb{R}^6/\mathbb{Z}^6$.
By abuse of notation, $\omega_0$, $J_0$, $g_0$, $\Omega_0$ denote the symplectic form, the complex structure,
the metric, the complex volume form, respectively on $\mathbb{T}^6$ as well as on $\mathbb{R}^6$.
Likewise, $B(R)$ denotes the ball of radius $R$ centred at $0$ in $(\mathbb{T}^6,g_0)$ as well as in $(\mathbb{R}^6,g_0)$.
Set $\mathbb{T}^3_1=\mathbb{R}^3_1/(\mathbb{R}^3_1\cap \mathbb{Z}^6)$, $\mathbb{T}^3_2=\mathbb{R}^3_2/(\mathbb{R}^3_2\cap \mathbb{Z}^6)$.
$\mathbb{T}^3_1$, $\mathbb{T}^3_2$ are special Lagrangian submanifolds of $(\mathbb{T}^6,\omega_0,J_0,\Omega_0)$.

In view of the proof of the theorem of D. Lee~\cite[Theorem~3]{D. Lee} or Joyce~\cite[Theorem~9.10]{Joyce},
for every $\delta>0$ there exist $(J_{\delta},\Omega_{\delta})$, $M_{\delta}$, $\epsilon_{\delta}$, $b_{\delta}$
and $b_{\delta}'$ with the following properties:
\begin{itemize}
\item [(P1)]
$J_{\delta}=f_{\delta}^*J_0$, $\Omega_{\delta}=f_{\delta}^*\Omega_0$
for some $f_{\delta}\in GL(3,\mathbb{R})$ such that $f_{\delta}^*\omega_0=\omega_0$;
\item [(P2)]
$M_{\delta}$ is a compact special Lagrangian submanifold of
$(\mathbb{T}^3,\omega_0,J_{\delta},\Omega_{\delta})$, i.e.,
a $\re\Omega_{\delta}$-submanifold of $(\mathbb{T}^6,g_{\delta})$, where $g_{\delta}=f_{\delta}^*g_0$;
\item [(P3)]
$\epsilon_{\delta}>0$, $0<b_{\delta}<b_{\delta}'<1$,
$\lim_{\delta \to 0}\epsilon_{\delta}=0$, $\lim_{\delta \to 0}b_{\delta}/{\delta}=\infty$;
\item [(P4)]
$M_{\delta}\cap B(b_{\delta} )$ is the graph of some normal vector field on ${\delta}L\cap B(b_{\delta})$
in $B(b_{\delta})$ with $C^1$-norm less than $b_{\delta}\epsilon_{\delta}$,
where ${\delta}L \cap B(b_{\delta})$ is embedded in $\mathbb{T}^6$;
\item [(P5)]
$M_{\delta}\setminus \overline{B(b_{\delta}' )}$ is the graph of some normal vector field on $(\mathbb{T}^3_1\cup \mathbb{T}^3_2)\setminus \overline{B(b_{\delta}')}$ in
$(\mathbb{T}^6\setminus \overline{B(b_{\delta}')},g_{\delta})$ with $C^1$-norm less than
$b_{\delta}\epsilon_{\delta}$;
\item [(P6)]
there exists a normal vector field $\nu$ on $(\mathbb{T}^3_1\cup \mathbb{T}^3_2)\cap A(b,b')$
in $(A(b,b'),g_{\delta})$ such that $M_{\delta}\cap A(b,b')=G(\nu)$ with $\|\nu\|_{C^1_{\mathrm{cyl}}}\leq \epsilon_{\delta}$ for some $b$, $b'$ with $0<b<b_{\delta}<b_{\delta}'<b'<1$
in the notation of Section~\ref{statement},
where $A(b,b')=B(b')\setminus \overline{B(b)}$;
\item [(P7)]
$M_{\delta}$ is diffeomorphic to the connected sum $\mathbb{T}^3_1 \# \mathbb{T}^3_2$.
\end{itemize}
The main result of this paper is the following:
\begin{theorem}
\label{uniqueness}
Let $(J_{\delta},\Omega_{\delta})$, $M_{\delta}$, $\epsilon_{\delta}$, $b_{\delta}$ and $b_{\delta}'$ be as above.
Let $M_{\delta}'$ be such that:
\begin{itemize}
\item[\upshape{(P2')}]
$M_{\delta}'$ is a compact special Lagrangian submanifold of
$(\mathbb{T}^6,\omega_0,J_{\delta},\Omega_{\delta})$;
\item[\upshape{(P4')}]
$M_{\delta}'\cap B(b_{\delta})$ is the graph of some normal vector field on ${\delta}L\cap B(b_{\delta})$
in $(B(b_{\delta}),g_{\delta})$ with $C^1$-norm less than $b_{\delta}\epsilon_{\delta}$;
\item[\upshape{(P5')}]
$M_{\delta}'\setminus \overline{B(b_{\delta}')}$ is the graph of some normal vector field on $(\mathbb{T}^3_1\cup \mathbb{T}^3_2)\setminus \overline{B(b_{\delta}')}$ in
$(\mathbb{T}^6\setminus \overline{B(b_{\delta}')},g_{\delta})$ with $C^1$-norm less than $b_{\delta}\epsilon_{\delta}$.
\end{itemize}
Then, $M_{\delta}'=M_{\delta}+t$ for some $t\in \mathbb{T}^6$ whenever $\delta$ is sufficiently small.
\end{theorem} 
The proof of Theorem~\ref{uniqueness} is divided into the following two propositions:
\begin{proposition}
\label{uniqueness1}
$M_{\delta}'$ is sufficiently close to $M_{\delta}$ in the $C^1$-topology induced by the metric $g_{\delta}$ on $\mathbb{T}^6$ whenever $\delta$ is sufficiently small. 
\end{proposition}
\begin{proposition}
\label{uniqueness2}
If a special Lagrangian submanifold $M_{\delta}''$ of $(\mathbb{T}^6,\omega_0,J_{\delta},\Omega_{\delta})$ is sufficiently close to $M_{\delta}$ in the $C^1$-topology induced by the metric $g_{\delta}$ on $\mathbb{T}^6$, then
$M_{\delta}''=M_{\delta}+t$ for some $t\in \mathbb{T}^6$ whenever $\delta$ is sufficiently small.
\end{proposition}
\begin{proof}[Proof of {Proposition~$\ref{uniqueness1}$}]
By (L4), (P3), (P4) and (P5), there exist $b_0$, $b_1$, $\beta$, $\nu_0$ and $\nu_1$ such that:
\begin{itemize}
\item $b_0>0$, $b_1>0$, $0<\beta <1$, $b_0<b_{\delta}<b_{\delta}'<b_1\beta$;
\item for each $i=0,1$, $\nu_i$ is a normal vector field on
$(\mathbb{T}^3_1\cup \mathbb{T}^3_2)\cap A(b_i\beta,b_i)$
in $A(b_i\beta,b_i)$;
\item $M_{\delta}\cap A(b_i\beta,b_i)=G(\nu_i)$ with $\|\nu_i\|_{C^1_{\mathrm{cyl}}}\leq \epsilon/2$,
for each $i=0,1$
in the notation of Section~$\ref{statement}$ whenever $\delta$ is sufficiently small;
\end{itemize}
here $A(b,b')=B(b')\setminus \overline{B(b)}$ for each $b$, $b'$ with $0<b<b'<1$,
and $\epsilon>0$ is as in Theorem~$\ref{theorem}$.
Therefore, by (P4') and (P5'),
there exist normal vector fields $\nu_i'$ on
$((\mathbb{T}^3_1\cup \mathbb{T}^3_2)\cap A(b_i\beta, b_i),g_{\delta})$ in $A(b_i\beta ,b_i)$
for each $i=0,1$ such that
\begin{equation*}
M_{\delta}'\cap A(\beta b_i,b_i)=G(\nu_i')\text{ with }\|\nu_i'\|_{C^1_{\mathrm{cyl}}}\leq \epsilon
\text{ for each }i=0,1
\end{equation*}
whenever $\delta$ is sufficiently small.
By (P1) and (P2'), therefore, $M_{\delta}'\cap A(b_0\beta,b_1)$ satisfies the assumption of Theorem~\ref{theorem}.
By Theorem~\ref{theorem}, therefore, there exists a normal vector field $\nu'$ on $(\mathbb{T}^3_1\cup \mathbb{T}^3_2) \cap A(b_0\beta,b_1)$ in $(A(b_0\beta, b_1),g_{\delta})$ such that
$M_{\delta}'\cap A(b_0\beta,b_1)=G(\nu')$ with $\|\nu'\|_{C^1_{\mathrm{cyl}}}\leq \epsilon_{\delta}'$
for some $\epsilon_{\delta}'$ converging to $0$ as $\delta\to 0$.
This, together with (P4), (P5), (P6), (P4') and (P5'), proves Proposition~\ref{uniqueness1}.
\end{proof}
\begin{proof}[Proof of {Proposition~$\ref{uniqueness2}$}]
Let $\mathcal{D}(M_{\delta})$ be the space of all special Lagrangian submanifolds of $(\mathbb{T}^6,\omega_0,J_{\delta},\Omega_{\delta})$ which are sufficiently close to $M_{\delta}$ in the $C^1$-topology induced by $g_{\delta}$ on $\mathbb{T}^6$.
By the deformation theory of Mclean~\cite[Section~3]{Mclean}, $\mathcal{D}(M_{\delta})$ is a manifold of dimension
$b^1(M_{\delta})$, where $b^1(M_{\delta})$ is the first Betti number of $M_{\delta}$.
By (P7), $b^1(M_{\delta})=6$.
Thus,
\begin{equation}\label{dimension6} \dim \mathcal{D}(M_{\delta})=6. \end{equation}
For each $t\in \mathbb{T}^6$, by (P1), $M_{\delta}+t$ is a special Lagrangian submanifold of $(\mathbb{T}^6,\omega_0,J_{\delta},\Omega_{\delta})$.
Therefore, $f:t\mapsto M_{\delta}+t$ maps a neighbourhood of $0\in\mathbb{T}^6$ into $\mathcal{D}(M_{\delta})$.
By (P5), if $\delta$ is sufficiently small,
then $df_0:T_0\mathbb{T}^6\rightarrow T_{f(0)}\mathcal{D}(M_{\delta})$ is one-to-one.
Therefore, by \eqref{dimension6} and Inverse Function Theorem, $f$ is a diffeomorphism of a neighbourhood of $0\in\mathbb{T}^6$ onto a neighbourhood of $f(0)=M_{\delta}$ in
$\mathcal{D}(M_{\delta})$.
This completes the proof of Proposition~\ref{uniqueness2}.
\end{proof}
\begin{proof}[Proof of {Theorem~$\ref{uniqueness}$}]
Proposition~\ref{uniqueness1} and Proposition~\ref{uniqueness2} imply Theorem~\ref{uniqueness}.
\end{proof}
\begin{remark}
The key step to the proof of Theorem~\ref{uniqueness} is
the proof of Proposition~\ref{uniqueness1}, where we have made a direct use of Theorem~\ref{theorem},
which assumes that the ambient space is a flat Euclidean space.
This is why we considered the flat torus $(\mathbb{T}^6,g_{\delta})$.
It is very likely, however, that modification of Theorem~\ref{theorem} leads to an extension of
Theorem~\ref{uniqueness} to more general Calabi--Yau manifolds.
The author plans to work it out in the near future.
\end{remark}

\end{document}